\documentclass{article}

\usepackage{amssymb, amsthm, amsmath}
\usepackage{xy} \input xy \xyoption{all}
\usepackage[usenames,dvipsnames]{color} 
\usepackage{hyperref} 
\hypersetup{colorlinks, citecolor=blue, filecolor=blue, linkcolor=blue, urlcolor=blue} 

\usepackage{graphicx}

\theoremstyle{plain}
\newtheorem{theorem}{Theorem}[section]
\newtheorem{lemma}[theorem]{Lemma}
\newtheorem{corollary}[theorem]{Corollary}
\theoremstyle{definition}

\newtheorem{definition}[theorem]{Definition}
\newtheorem{example}[theorem]{Example}
\newtheorem{remark}[theorem]{Remark}
\newtheorem{algorithm}[theorem]{Algorithm}

\makeatletter
\newcommand{\dashedrightarrow}[1][2pt]{%
  \settowidth{\@tempdima}{$\longrightarrow$}\longrightarrow
  \makebox[-\@tempdima]{\hskip-1.5ex\color{white}\rule[0.5ex]{#1}{1pt}}
  \phantom{\longrightarrow}
}
\makeatother

\newcounter{mnotecounter}

\newcommand{\cP}{\mathcal{P}}
\newcommand{\tcP}{\widetilde{\mathcal{P}}}
\newcommand{\Image}{\mathrm{Image}}

\begin{document}

\title{Covering Rational Ruled Surfaces}

\author{J. Rafael Sendra, Carlos Villarino\\
       {Dept. of Physics and Mathematics, U. of Alcal\'a}\\
       {Ap. Correos 20}\\
       {E-28871 Alcal\'a de Henares (Madrid, Spain)}\\
       {$\{$Rafael.Sendra,Carlos.Villarino$\}$@uah.es}
\\ \\
       David Sevilla\\
       {U. Center of M\'erida, U. of Extremadura}\\
       {Av. Santa Teresa de Jornet 38}\\
       {E-06800 M\'erida (Badajoz, Spain)}\\
       {sevillad@unex.es}}

\date{}

\maketitle


\begin{abstract}
We present an algorithm that covers any given rational ruled surface with two rational parametrizations. In addition, we present an algorithm that transforms any rational surface parametrization into a new rational surface parametrization without affine base points and such that the degree of the corresponding maps is preserved.
\end{abstract}

2010 Mathematics Subject Classification: 14Q10, 68W30.

Keywords: parametrization of ruled surfaces, normality, base points

\section{Introduction}
One of the most important features of rational varieties, at least in practice, is the possibility to choose between parametric or implicit representations depending on the nature of the problem one is dealing with; examples are the computation of intersections, plotting figures, line and surface integrals, etc. However, when using the parametric representations additional difficulties may appear, and the feasibility of the strategy is affected. In particular, if the parametrization is not surjective, some solving strategies may fail.

In \cite{Ricam}, Example 1 illustrates a situation where the computation of the intersection of two surfaces fails when a non surjective parametrization is used. Let us see another motivating example.

\begin{example}
The Hausdorff distance appears naturally in applications in computer aided design, pattern matching and pattern recognition (see e.g. \cite{Bai}, \cite{Chen}, \cite{Kim}), when measuring the resemblance between two geometric objects. The computation or estimation of the Hausdorff distance implies, in particular, measuring the distance of a point to a set. Let us assume that we want to measure the distance of the point $A=(4/5, 6/5, 1)$ to the surface $S$ defined by $f(x,y,z)=xy-2yz+z^2$. Applying Lagrange multipliers one gets that the distance of $A$ to $S$ is $\sqrt{2}/5=0.283\ldots$ and it is reachable at $B=(1,1,1)\in S$. Nevertheless, in general, approaching this problem using implicit equations turns to be computationally intractable. Instead, one can try to use a parametrization of the surface so that the problem reduces to a  optimization problem without constrains. In our case, $S$ is rational, indeed it is rational ruled surface, and can be parametrized as
 \[ \cP(s,t)=\left((s^2-1)t, s^2t, t(s^2+s)\right). \]
However, if we optimize the function $\|A-\cP(s,t)\|^2$ we find that the minimum is obtained at $(0.889\ldots, -0.042\ldots, 0.155\ldots)$ and the distance is then estimated as $1.504\ldots$. The problem is that $B\in S\smallsetminus \Image(\cP)$, so it cannot be found with the parametrization. Nevertheless, $\cP$ satisfies the hypothesis in Theorem \ref{theorem-recta-critica}, and therefore we can determine that the $S\smallsetminus \Image(\cP)$ is included in the line $(t,t,t)$. Thus, we now optimize the  $\|A-(t,t,t)\|^2$ to get $B$ as solution.
\end{example}

In the case of curves, non-surjectivity is not so important since every rational proper parametrization of a curve may miss at most one point that can be easily computed (see e.g. \cite{AndradasRecio}, \cite{Sendra2002a}). The situation changes when working with rational parametrizations of algebraic surfaces: the missing subset can be of dimension 1.

Some authors have addressed the problem of finding surjective parametrizations of rational surfaces; see \cite{Bajaj}, \cite{Gao} for the case of quadrics or \cite{Ricam} for certain particular types of rational surfaces. Alternatively, one can compute finitely many rational parametrizations such that the union of their images covers the whole surface. This was done for the real general case, in \cite{Bajaj}, by computing a cover with $2^n$ parametrizations, where $n$ is the dimension of the rational variety; i.e. in the surface case, with four pieces. In \cite{issac} we show that, if a surface admits a rational parametrization without projective base points, then it can be covered with at most three pieces. Continuing with this research, in this paper we analyze the problem of covering   rational ruled surfaces. The next example shows that for the same surface, changing the parametrization, can make the missing subset bigger.

\begin{example}
We consider the ruled surface $S$ given by $x^2y-2xy^2+2y^3-3y^2z+3yz^2-z^3=0$. $S$ can be parametrized in {\it ruled form} as
\[ \cP(s,t)= \left(\frac{(s^3-1)t}{s(s+1)}, \frac{s^2t}{s+1}, \frac{(s^2+1)t}{s+1}\right) \]
Applying the algorithm in \cite{Ricam} for computing the critical sets, we obtain that $\cP$ covers all the surface but the three lines $\{x=2y=z\}, \{x=y=z\}, \{y=z=0\}$. However, the reparametrization
\[ \cP(s, ts(s+1))= \left( (s^2-1)t, ts^2, (s^2+s)t \right), \]
that is also in {\it ruled form},  only misses the line $\{x=y=z\}$ (see Theorem \ref{theorem-recta-critica}).
\end{example}

In this paper we prove that a rational ruled surface can always be covered with two rational surface parametrizations in {\it ruled form}. More precisely, we prove that there always exists a rational parametrization that, at most, misses a line on the surface; then the second parametrization covers that line. In order to compute the first parametrization we need parametrizations without affine base points. Later we consider this problem in general, and we present an algorithm that transforms any rational surface parametrization into a new parametrization without affine base points.

\section{Covering Ruled Surfaces: Main Results}\label{sec:covering-theory}

In the sequel, we show that every rational ruled surface can be covered by means of, at most, two rational parametrizations.


\begin{definition}\label{def-ruled-param}
 A \emph{standardized ruled surface parametrization} of a ruled suface $S$ is a triple of rational functions that determines a dominant rational map
 \[\begin{array}{rccc}
  \cP: & k^2 & \dashedrightarrow & S \\ [1em]
   & (s,t) & \mapsto & \left(\displaystyle \frac{r_1(s)+t\cdot p_1(s)}{q(s)}, \frac{r_2(s)+t\cdot p_2(s)}{q(s)}, \frac{r_3(s)+t\cdot p_3(s)}{q(s)} \right)
 \end{array}\]
 such that those $p_i$ that are nonzero have the same degree and do not have any common root (note that not all three of them are zero).
\end{definition}

\begin{remark}\label{remark-pi}
In a standardized ruled surface parametrization, if two of the polynomials $p_i$ are zero, then the third has to be a nonzero constant. In addition, we observe that, in that case, say $p_1=p_2=0$, then $\cP(s,(-r_3+qt)/p_3)=(r_1/q,r_2/q,t)$. So, $S$ is a cylinder over the plane curve $(r_1/q,r_2/q,0)$ and hence, applying the results in \cite{Sendra2002a}, $S$ can be parametrized surjectively.
\end{remark}


\begin{lemma}\label{lemma-stardard}
 Every rational ruled surface admits a standardized ruled surface parametrization.
\end{lemma}

\begin{proof}
By \cite{Sonia}, every rational ruled surface admits a rational parametrization of the form
\[ \cP(s,t)=\left(\frac{\alpha_1(s) +t \beta_1(s)}{\gamma(s)},\frac{\alpha_2(s) +t \beta_2(s)}{\gamma(s)}, \frac{\alpha_3(s) +t \beta_3(s)}{\gamma(s)}\right). \]
If two $\beta_i$ are zero, say $\beta_1=\beta_2=0$, then $\cP(s,t/\beta_3(t))$ is standardized. Let us suppose that at least two $\beta_i$ are nonzero. Then, we can assume that  those components of $\cP(s,t)$ depending on $t$ also do depend on $s$; if this is not the case a suitable change of the form $(s,as+bt)$ provides a parametrization with this property. Furthermore, applying a transformation of the form $(\frac{as+b}{cs+d},t)$,  we can assume that all nonzero $\beta_i$ have the same degree. It only remains to ensure that the gcd of the polynomial coefficients of $t$ are coprime. But this can be achieved by performing the transformation $(s, t/\Delta(s))$, where $\Delta$ is the gcd of the nonzero $\beta_i$.
\end{proof}

Associated to the standardized ruled surface parametrization $\cP(s,t)$, we consider the polynomials
\begin{equation}\label{eq-H}
 \begin{array}{l} H_1=r_1(s)+t\cdot p_1(s)-x\cdot q(s), \\ H_2=r_2(s)+t\cdot p_2(s)-y\cdot q(s), \\
H_3=r_3(s)+t\cdot p_3(s)-z\cdot q(s),
\end{array}
\end{equation}
as well as the polynomials $A_{ij} = p_iH_j-p_jH_i\in k[x,y,z,s]$ for $i\neq j$. We express $A_{ij}$ as
\begin{equation}\label{eq-A}
\begin{array}{r}
A_{12}= qp_2x -qp_1y -\alpha_{12}, \\ A_{13}= qp_3x -qp_1z -\alpha_{13}, \\
A_{23}= qp_3y -qp_2z -\alpha_{23}, \\
\alpha_{ij}=-p_{i}r_j+p_jr_i.
\end{array}
\end{equation}
 We have the following lemma.

\begin{lemma}\label{lemma-crit}
Let $\cP$ be a standardized ruled surface parametrization without affine base points of a surface $S$. Then $S\smallsetminus\Image(\cP)$ is contained in the variety $\cal W$ defined by $\{{\rm LC}_s(A_{ij})\}_{i\neq j}$, where ${\rm LC}_s$ denotes the leading coefficient w.r.t. $s$.
\end{lemma}
\begin{proof}
 In the ring $k[x,y,z,s,t,w]$ we consider the ideal
 \[
  I = (H_1(s,t,x),\ H_2(s,t,y),\ H_3(s,t,z),\ w\cdot q(s)-1).
 \]
Then $\Image(\cP)=\pi(V(I))$ where $\pi(x,y,z,s,t,w)=(x,y,z)$. We will use the extension theorem (see e.g. Chp.3, Th 3, p. 115 in \cite{CoxLittleOshea2007a}) to determine which points $(x,y,z)\in S$ can be lifted to $V(I)$. To this end we define
 \[
  I_1 = I \cap k[x,y,z,s,t], \quad I_2 = I \cap k[x,y,z,s], \quad I_3 = I \cap k[x,y,z].
 \]
 \begin{itemize}
  \item Extension from $I_1$ to $I$: a point $(x_0,y_0,z_0,s_0,t_0)$ has an extension provided $q(s_0)\neq0$. But if $q(s_0)=0$ we see from the equations that $r_i(s_0)+t_0p_i(s_0)=0$ for all $i$, and $(s_0,t_0)$ would be a base point, contrary to the hypotheses.
  \item Extension from $I_2$ to $I_1$: in order to extend a point $(x_0,y_0,z_0,s_0)$ to the coordinate $t$ it suffices that $p_1,p_2,p_3$ do not simultaneously vanish at $s_0$. This always holds since by definition they have no common root. Note that if two of the $p_i$ are zero, the other is a nonzero constant, and the extension is possible.
  \item Extension from $I_3$ to $I_2$:
   a point $(x,y,z)$ can be extended to the coordinate $s$ if for at least one of the polynomials $A_{ij}$ the leading coefficient in $s$ does not vanish at the point.
 \end{itemize}
\end{proof}

\begin{lemma}\label{lemma-recta}
The variety $\cal W$ introduced in Lemma \ref{lemma-crit} is either empty or a line. Furthermore, $\mathcal{W}=\emptyset$ if and only if $\deg(\alpha_{ij})>\deg(p_k q)$, for some different $i,j\in \{1,2,3\}$ and nonzero $p_k$.
\end{lemma}
\begin{proof}
Let us assume $p_1\neq 0$. If $\deg(\alpha_{ij})>\deg(p_1 q)$, for some different $i,j\in \{1,2,3\}$, then ${\rm LC}_s(A_{ij})$ is a nonzero constant and $\mathcal{W}=\emptyset$.

If $\deg(\alpha_{ij})\leq \deg(p_1 q)$ for all $i\neq j$, we distinguish two cases. If $p_2=p_3=0$ then $A_{12}=-p_1qy-\alpha_{12}, A_{13}=-p_1qz-\alpha_{13}, A_{23}=0$. Then, $\cal W$ is defined by two linear polynomials, one depending on $y$ and the other on $z$. So $\cal W$ is a line. In the second case, let us assume that at least two $p_i$ are nonzero. Since $p_3 A_{12}-p_2 A_{13}+p_1 A_{23}=0$, then
\[ {\rm LC}_s(p_1 A_{23})={\rm LC}_s(p_1){\rm LC}_s(A_{23})={\rm LC}_s(-p_3A_{12}+p_2 A_{13}). \]
Let us see that
\[ {\rm LC}_s(p_1){\rm LC}_s(A_{23})=-{\rm LC}_s(p_3){\rm LC}_s(A_{12})+{\rm LC}_s(p_2){\rm LC}_s(A_{13}). \]
If either $p_2$ or $p_3$ is zero, the result is clear. So, let none of them be zero. Then, $\deg_s(p_3A_{12})=\deg_s(p_2A_{13})$. Since the leading coefficient of $p_3A_{12}$ does depend on $\{x,y\}$ and the leading coefficient of $p_2A_{13}$ does depend on $\{y,z\}$,
$\deg_s(-p_3A_{12}+p_2 A_{13})=\deg_s(p_3A_{12})=\deg_{s}(p_2A_{13})$, from where the above equality on the leading coefficients follows.
In this situation we get that $\cal W$ is defined by ${\rm LC}_s(A_{12}),{\rm LC}_s(A_{13})$. Now, the result follows by taking into account that the rank of the linear system $\{{\rm LC}_s(A_{12})=0={\rm LC}_s(A_{13})\}$ is 2.
\end{proof}

Using the previous results one gets the following theorem.

\begin{theorem}\label{theorem-recta-critica}
 Let $\cP$ be a standardized ruled surface parametrization without affine base points of a surface $S$. Then $S\smallsetminus\Image(\cP)$ is contained in a line. Furthermore,
\begin{enumerate}
\item if there exists $i,j\in \{1,2,3\}$ such that $i\neq j$ and $\deg(\alpha_{ij})>\deg(p_k q)$ for nonzero $p_k$, then $\cP(s,t)$ is normal.
\item if for all $i,j\in \{1,2,3\}$, with $i\neq j$,  $\deg(\alpha_{ij})\leq \deg(p_k q)$ for nonzero $p_k$, then $S\smallsetminus\Image(\cP)$ is included in the line $V({\rm LC}_s(A_{12}),{\rm LC}_s(A_{13}),{\rm LC}_s(A_{23}))$.
\end{enumerate}
\end{theorem}

In Example \ref{ejemplo-grado-3}, one can see that the parametrization covers all the line $\cal W$ but a point, while in Example \ref{ejemplo-grado-5}, the parametrization only covers two points on the line $\cal W$.

In the previous theorem we have imposed the condition of not having affine base points. Let us see that this is always achievable.


\begin{lemma}\label{lemma-reglada-sin-puntos-base}
Every standardized ruled surface parametrization can be re\-pa\-ra\-me\-tri\-zed into another one without affine base points and where the degree of the induced map is preserved.
\end{lemma}

\begin{proof}
Let us assume first that all the $p_i$ are nonzero. Let $f(s)$ be a polynomial such that $f(s_1)=t_1$ for some base point $(s_1,t_1)$. With the change $(s,1/t+f(s))$ the resulting parametrization is
 \[ \left( \frac{Q_i(s)\cdot t+p_i(s)}{t\cdot q(s)} \right)_{i=1,2,3} \qquad \mbox{where $Q_i(s)=r_i(s)+f(s)p_i(s)$}. \]
 Since $s_1$ is a common root of $q$ and the $Q_i$, if we define $\widetilde{Q}=\gcd(Q_1,Q_2,Q_3,q)$, we have $\deg(\widetilde{Q})\geq1$. Now with the change $(s,1/(\widetilde{Q}t))$ we obtain the new parametrization
 \[ \left( \frac{Q_i/\widetilde{Q} + p_i\cdot t}{q/\widetilde{Q}} \right)_{i=1,2,3}. \]
 Note that this is standardized as well, but the degree of the denominator is strictly smaller than the original. Therefore repeating this procedure finitely many times we obtain a standardized parametrization without affine base points (since that is the case when the denominator is a constant). Finally, note that all transformations considered are birational, and hence the degrees of the maps are preserved.

 If any $p_i=0$, the corresponding component of the parametrization does not change after the first reparametrization, resulting in $Q_i=r_i$. The second reparametrization does not change the component as well, but the common factor $\widetilde{Q}$ of $r_i$ and $q$ can be directly simplified in that fraction.
\end{proof}

\begin{remark}\label{remark-pto-bases}
 In the previous result we can have some control on the removal of base points that occurs effectively in each iteration.
 Namely, suppose that the zeros of $q$ are $s_1,\ldots,s_l$ where $s_1,\ldots,s_k$, $k\leq l$, are the first coordinates of the base points of $\cP$. Note that for each of $s_1,\ldots,s_k$ there is exactly one base point $(s_i,t_i)$.

 Let $f(s)$ be an interpolating polynomial of $$(s_1,t_1),\ldots,(s_k,t_k), (s_{k+1},0),\ldots,(s_l,0).$$ As before, we define $Q_i(s)=r_i(s)+f(s)p_i(s)$ and $\widetilde{Q}=\gcd(Q_1,Q_2,Q_3,q)$, and make the change $(s,\widetilde{Q}t+f(s))$ to obtain
 \[ \cP^{(1)} = \left( \frac{r_i^{(1)} + p_i\cdot t}{q^{(1)}} \right)_{i=1,2,3} \qquad \mbox{where $r_i^{(1)}=Q_i/\widetilde{Q}$, \ $q^{(1)}=q/\widetilde{Q}$}. \]
 Note that the roots of $\widetilde{Q}$ are precisely $s_1,\ldots,s_k$.
 We will show that $\cP^{(1)}$ has at most as many base points as the number of multiple roots of $q$ among $s_1,\ldots,s_k$. To this end let $(\alpha,\beta)$ be a base point of $\cP^{(1)}$. Since $\alpha$ is a root of $q^{(1)}$, it must be a multiple root of $q$, say $\alpha=s_i$. If $i>k$ then $Q_i(\alpha)=r_i(\alpha)$. Now, by definition of $\beta$, we have $\beta=-r_i^{(1)}(\alpha)/p_i(\alpha)$ for some $i$. But then the point $(\alpha,\beta\widetilde{Q}(\alpha))$ is a base point of $\cP$, contradiction.
\end{remark}

\begin{corollary}\label{cor+2+1}
Every rational ruled surface can be parametrized in an standardized way that misses at most one line.
\end{corollary}

\begin{theorem}\label{theorem+2+2}
Every rational ruled surface can be covered with at most two surface parametrizations.
\end{theorem}
\begin{proof} By Lemma \ref{lemma-reglada-sin-puntos-base} we can assume that we are given an standardized  pa\-ra\-me\-tri\-za\-tion $\cP$ without affine base points. We use for $\cP$ the notation in Definition \ref{def-ruled-param} and in the previous results. By Theorem \ref{theorem-recta-critica}, we can also assume that $\max\{\deg(\alpha_{12}),\deg(\alpha_{13}),\deg(\alpha_{23})\}\leq \deg(p_kq)$ for nonzero $p_k$.

First we assume that all $p_i(s)$ are nonzero. Consider the reparametrizations
\[ \mathcal{Q}(s,t)=\cP\left(s,\frac{qt-r_3}{p_3}\right)=\left(\frac{q p_1t+\alpha_{13}}{p_3q},\frac{qp_2t+\alpha_{23}}{p_3q},t \right) \]
and
\[ \mathcal{H}(s,t)=\mathcal{Q}\left(\frac{1}{s},t\right).\]
Because of our above degree assumptions, we know that the degrees in $s$ of the numerator and denominator of each (first and second) component of $\mathcal{Q}$ are equal. Therefore, $s$ is not a factor of the denominators in $\mathcal{H}(s,t)$. So, $\mathcal{H}(0,t)$ is well defined and, indeed,
\[ \mathcal{H}(0,t)=\left(\frac{{\rm LC}_s(q p_1t+\alpha_{13})}{{\rm LC}_s(p_3q)},\frac{{\rm LC}_s(qp_2t+\alpha_{23})}{{\rm LC}_s(p_3q)},t \right) \]
that parametrizes the line $\cal W$.

A similar argument with obvious modifications works in the case when some $p_i$ are zero.
\end{proof}

\section{Covering Ruled Surfaces: Algorithm and Examples}\label{sec:covering-algorithm}

In order to derive an algorithm from the previous results, we need to algorithmically show how to remove the affine base points of an standardized ruled parametrization. This, essentially, requires to compute interpolation polynomials (see proof of Lemma \ref{lemma-reglada-sin-puntos-base} and Remark \ref{remark-pto-bases}). In the following lemma we see how to actually compute the interpolation polynomial  without explicitly determining the coordinates of the base points; i.e. without approximating roots.

\begin{lemma}\label{lemma-interpolation}
Let $\cP(s,t)$ be an standardized ruled parametrization as in Def. \ref{def-ruled-param} with affine base points. Let $I$ be the ideal generated by $\{p_1t+r_1,p_2t+r_2,p_3t+r_3,q\}$ in $k[s,t]$. Then,  there exists a polynomial of the form $t-f(s)$ in $\sqrt{I}$ where $f(s)$  interpolates the affine base points of $\cP(s,t)$.
\end{lemma}
\begin{proof} As observed in Remark \ref{remark-pto-bases}, all affine base points of $\cP$ have different $s$-coordinate. Thus, there exists an interpolating polynomial $f(s)$ passing through all base points. So,
 $t-f(s)$ vanishes on all the points in the variety of $I$. So, $t-f(s)\in \sqrt{I}$.
\end{proof}

Now, we are ready to outline our algorithm.

\begin{algorithm}\label{alg-1}
Given a rational parametrization $\cP(s,t)$ of a ruled surface $\cal S$, the algorithm computes a covering of $\cal S$.
\begin{enumerate}
\item If $\cP$ is not of the form $((r_i(s)+p_i(s)t)/q(s))_{i=1,2,3}$ apply the algorithm in \cite{Sonia} and replace $\cP$.
\item If $\cP$ is not in standardized form (see Def. \ref{def-ruled-param}) do the following
\begin{enumerate}
\item If some of the numerators of $\cP$ does not depend on $s$, replace $\cP$ by $\cP(s, as +
bt)$ with $a,b\in k$.
\item If the polynomials $p_1,p_2,p_3$ do not have the same degree, replace $\cP$ by $\cP((as+b)/(cs+d), t)$ where $a,b,c,d\in k$ and $ad-bc\neq 0$.
\item Replace $\cP$ by reparametrization $\cP(s, t/\Delta(s))$ where $\Delta$ is the gcd of the nonzero $p_i$.
\end{enumerate}
\item Calculate $\sqrt{I}$,  where $I$ is the ideal generated by $\{p_1t+r_1,p_2t+r_2,p_3t+r_3,q\}$ in $k[s,t]$. This can be done with a relatively inexpensive Gr\"obner basis computation (see e.g.  Ex 2.3.23 and 24 in \cite{Adams} and \cite{Seidenberg}).
 \item Calculate a Gr\"obner basis of $\sqrt{I}$ with respect to the lexicographical ordering $t>s$.
 \begin{enumerate}
  \item If the basis does not contain a polynomial of the form $t-f(s)$, by elementary properties of Gr\"obner basis it follows that there is no polynomial of that form in $\sqrt{I}$, so by  Lemma \ref{lemma-interpolation} we know that $\cP$ does not have affine base points.
  \item In the other case, let $t-f(s)$ belong to the basis, do
  \begin{enumerate}
  \item Replace $\cP$ by  $\cP(s, 1/t + f(s))$.
  \item Let $\widetilde{Q}$ be the gcd of the coefficients of $t$ of the numerators of $\cP$ and $q$, then replace $\cP$ by
  $\cP(s, 1/(\widetilde{Q}t))$.
  \item Repeat Steps 3 and 4 while $\sqrt{I}$ has an element of the form $s-f(t)$.
  \end{enumerate}
 \end{enumerate}
\item Compute the polynomials $\alpha_{ij}$  (see (\ref{eq-A})).
 \item If there exist $i,j\in \{1,2,3\}$ such that $i\neq j$ and $\deg(\alpha_{ij}) > \deg(p_kq)$ for nonzero $p_k$, RETURN $\cP(s,t)$.
 \item Assume that $p_k\neq 0$, compute $$\mathcal{H}(s,t) = \cP\left(\frac{1}{s},\frac{q(1/s)t-r_k(1/s)}{p_k(1/s)}\right)$$
and RETURN $[\cP(s,t),\mathcal{H}(s,t)]$.
\end{enumerate}
\end{algorithm}

\begin{remark}
 Remark \ref{remark-pto-bases} shows that, in general, the number of iterations of the loop in Step 4 is small. Indeed it is bounded by the maximum multiplicity of the roots of the denominator $q(t)$ of $\cP$.

In addition we observe that all parametrizations in the output of the algorithm are of ruled form, that is, of the form $(\alpha_1(s),\alpha_2(s),\alpha_3(s))+$\linebreak $t(\beta_1(s),\beta_2(s),\beta_3(s))$.
\end{remark}

Let us illustrate Algorithm \ref{alg-1} by some examples.

\begin{example}\label{ejemplo-grado-3}
We consider the parametrization
\[  \cP(s,t) =\left( \frac{r_1(s)+t p_1(s)}{q(s)}, \frac{r_2(s)+t p_2(s)}{q(s)}, \frac{r_3(s)+t p_3(s)}{q(s)} \right) = \]
\[
\left(  {\frac {t \left( {s}^{2}+s+1 \right) +s}{s \left( s-1 \right) }},{
\frac {t \left( {s}^{2}+2\,s \right) +s}{s \left( s-1 \right) }},{
\frac {t \left( {s}^{2}+1 \right) +s}{s \left( s-1 \right) }}\right).
\]
It parametrizes the degree 3 ruled surface defined by

\noindent $F(x,y,z)=5\,{x}^{3}-9\,{x}^{2}y-8\,{x}^{2}z+5\,x{y}^{2}+11\,xyz+3\,x{z}^{2}-{y}
^{3}-3\,{y}^{2}z-3\,y{z}^{2}-4\,{x}^{2}+4\,xy+4\,xz-{y}^{2}-2\,yz-{z}^
{2}.$

We observe that $\cP(s,t)$ is in standardized form. So, we go to Step 3 in Algorithm \ref{alg-1}. $I$ is the ideal generated by
 $\{p_1t+r_1,p_2t+r_2,p_3t+r_3,q\}$ in $k[s,t]$. We get $\sqrt{I}=I$, and a Gr\"obner basis w.r.t. the lexicographic ordering $t>s$ (Step 4) is
 $\{s,t\}$. So, in Step 4 (b) we get that $f(s)=0$; note that the origin is the only affine base point. In Step 4 (b, i), we replace $\cP(s,t)$ by
 $\cP(s,1/t+0)$, namely
 \[ \cP(s,t)=\left({\frac {{s}^{2}+st+s+1}{ts \left( s-1 \right) }},{\frac {s(s+2+t)}{ts
 \left( s-1 \right) }},{\frac {{s}^{2}+st+1}{ts \left( s-1 \right) }}\right). \]
 In Step 4 (b, ii), $\widetilde{Q}=\gcd(s,s,s,s(s-1))=s$. So, we replace $\cP$ by $\cP(s, 1/(st))$, namely
 \begin{equation}\label{eq-ex-1}
  \cP(s,t)=\left( {\frac {{s}^{2}t+st+t+1}{s-1}},{\frac {{s}^{2}t+2\,st+1}{s-1}},{
\frac {{s}^{2}t+t+1}{s-1}}\right).
\end{equation}
Now, the lexicographic order Gr\"obner basis of $\sqrt{I}$ is $\{1\}$, hence $\cP$ does not have base points. In Step 5 we get
\[ \alpha_{12}=s-1,\alpha_{13}=-s,\alpha_{23}=2\,s-1. \]
In Step 6 the boolean conditions do not hold. In Step 7 we calculate the parametrization
\[ \mathcal{H}=\cP\left(\frac{1}{s},\frac{q(1/s)t-r_3(1/s)}{p_3(1/s)}\right)= \]
\[ \left({\frac {{s}^{3}t-2\,{s}^{3}-{s}^{2}-2\,s-t}{ \left( {s}^{2}+1
 \right)  \left( s-1 \right) }},-{\frac {{s}^{3}-2\,{s}^{2}t+2\,{s}^{2
}+st+2\,s+t}{ \left( {s}^{2}+1 \right)  \left( s-1 \right) }},{\frac {
st-2\,s-t}{s-1}}\right).\]
The algorithm returns the covering $[\cP(s,t),\mathcal{H}(s,t)]$ where $\cP$ is the parametrization in (\ref{eq-ex-1}).
\begin{figure}[ht]
 \centering
 \includegraphics[scale=0.7]{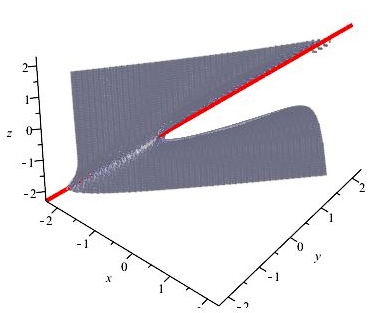}
 \caption{The surface in Example \ref{ejemplo-grado-3} and line $(t,t,t)$.}
 \label{fig: ex-1}
\end{figure}
Continuing with the example, since $\cP$ in (\ref{eq-ex-1}) is an standardized ruled parametrization without affine base points, by Theorem \ref{theorem-recta-critica}, the possible missing points of $\cP$ are included in the line defined by $\{x-y=0, x-z=0, y-z=0\}$, that is, the line $(t,t,t)$; see Fig. \ref{fig: ex-1}. In fact, $\cP$ covers all the line except the origin, by taking $\cP(s,0)$. Nevertheless,  ${\cal H}$ covers the whole line by taking ${\cal H}(0,t)$.
\end{example}

\begin{example}\label{ejemplo-grado-5}
We consider the parametrization
\[  \cP(s,t) =\left( \frac{r_1(s)+t p_1(s)}{q(s)}, \frac{r_2(s)+t p_2(s)}{q(s)}, \frac{r_3(s)+t p_3(s)}{q(s)} \right) = \]
\[
\left(  {\frac {t{s}^{3}+2\,{s}^{2}+1}{{s}^{2}-1}},{\frac {t \left( {s}^{3}+2
 \right) +s+1}{{s}^{2}-1}},{\frac {t \left( {s}^{3}+s+1 \right) +1}{{s
}^{2}-1}}\right).
\]
It parametrizes the degree 5 ruled surface defined by

\noindent $ F(x,y,z)=9\,{x}^{5}-45\,{x}^{4}y-24\,{x}^{4}z+77\,{x}^{3}{y}^{2}+80\,{x}^{3}yz-
15\,{x}^{3}{z}^{2}-83\,{x}^{2}{y}^{3}+34\,{x}^{2}{y}^{2}z-147\,{x}^{2}
y{z}^{2}+78\,{x}^{2}{z}^{3}+66\,x{y}^{4}-60\,x{y}^{3}z-189\,x{y}^{2}{z
}^{2}+460\,xy{z}^{3}-236\,x{z}^{4}+16\,{y}^{5}-86\,{y}^{4}z+111\,{y}^{
3}{z}^{2}+118\,{y}^{2}{z}^{3}-332\,y{z}^{4}+168\,{z}^{5}-104\,{x}^{4}+
319\,{x}^{3}y+108\,{x}^{3}z-207\,{x}^{2}{y}^{2}-621\,{x}^{2}yz+452\,{x
}^{2}{z}^{2}+147\,x{y}^{3}-382\,x{y}^{2}z+1034\,xy{z}^{2}-848\,x{z}^{3
}+297\,{y}^{4}-1549\,{y}^{3}z+3390\,{y}^{2}{z}^{2}-3380\,y{z}^{3}+1344
\,{z}^{4}+304\,{x}^{3}-741\,{x}^{2}y+389\,{x}^{2}z-267\,x{y}^{2}+1761
\,xyz-2314\,x{z}^{2}-4\,{y}^{3}+922\,{y}^{2}z-1930\,{z}^{2}y+1816\,{z}
^{3}-70\,{x}^{2}+597\,yx-2060\,zx+748\,{y}^{2}-1703\,zy+2940\,{z}^{2}-
761\,x-62\,y+2085\,z+746.$

We observe that $\cP(s,t)$ is in standardized form, so we go to Step 3 in Algorithm \ref{alg-1}. $I$ is the ideal generated by
 $\{p_1t+r_1,p_2t+r_2,p_3t+r_3,q\}$ in $k[s,t]$. We get $\sqrt{I}=I$, and a Gr\"obner basis w.r.t. the lexicographic ordering $t>s$ (Step 4) is
 $\{1\}$. Thus $\cP(s,t)$ does not have affine base points and we go to Step 5 to get
\[ \alpha_{12}=2{s}^{5}-{s}^{4}+4{s}^{2}+2 ,
\alpha_{13}=2{s}^{5}+2{s}^{3}+2{s}^{2}+s+1,
\alpha_{23}=-{s}^{4}-{s}^{2}-2s+1. \]
In Step 6 the boolean conditions do not hold. In Step 7 we get the parametrization
\[ \mathcal{H}=\cP\left(\frac{1}{s},\frac{q(1/s)t-r_3(1/s)}{p_3(1/s)}\right)= \]
\[ \left(-{\frac {{s}^{5}+{s}^{4}+2\,{s}^{3}-t{s}^{2}+4\,{s}^{2}+t+2}{ \left(
{s}^{3}+{s}^{2}+1 \right)  \left( {s}^{2}-1 \right) }},\right.\]\[ \left. {\frac {2\,{s}^
{5}t-3\,{s}^{5}-2\,{s}^{4}-2\,t{s}^{3}-{s}^{3}+t{s}^{2}-2\,{s}^{2}-s-t
}{ \left( {s}^{3}+{s}^{2}+1 \right)  \left( {s}^{2}-1 \right) }},{
\frac {t{s}^{2}-2\,{s}^{2}-t}{{s}^{2}-1}}\right).\]
The algorithm returns the covering $[\cP(s,t),\mathcal{H}(s,t)]$.
\begin{figure}[ht]
 \centering
 \includegraphics[scale=0.5]{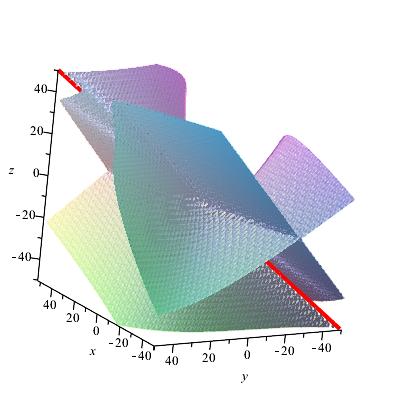}
 \caption{The surface in Example \ref{ejemplo-grado-5} and line $(t+2,t,t)$.}
 \label{fig: ex-2}
\end{figure}

Next, since the input parametrization $\cP$  is an standardized ruled parametrization without affine base points, by Theorem \ref{theorem-recta-critica}, the possible missing points of $\cP$ are included in the line defined by $\{x-y-2, x-z-2, y-z\}$, that is, the line $(t+2,t,t)$; see Fig. \ref{fig: ex-2}. In fact, on the line, $\cP$ reaches only the points
\[ \cP\left(\frac{1}{2}(1-i\sqrt{7}),\frac{1}{4}(3+i\sqrt{7})\right)=\left( {\frac {3}{32}}\,i\sqrt {7}+{\frac {65}{32}},{\frac {3}{32}}\,i\sqrt
{7}+\frac{1}{32},{\frac {3}{32}}\,i\sqrt {7}+\frac{1}{32}  \right),\]
 \[\cP\left(\frac{1}{2}(1+i\sqrt{7}),\frac{1}{4}(3-i\sqrt{7})\right)=\left(-{\frac {3}{32}}\,i\sqrt {7}+{\frac {65}{32}},-{\frac {3}{32}}\,i
\sqrt {7}+\frac{1}{32},-{\frac {3}{32}}\,i\sqrt {7}+\frac{1}{32} \right) \]
 Nevertheless,  ${\cal H}$ covers the whole line by taking ${\cal H}(0,t)$.
\end{example}

\begin{example}\label{ejemplo-grado-4}
We consider the parametrization
\[  \cP(s,t) =\left( \frac{r_1(s)+t p_1(s)}{q(s)}, \frac{r_2(s)+t p_2(s)}{q(s)}, \frac{r_3(s)+t p_3(s)}{q(s)} \right) = \]
\[
\left(  {\frac {ts+ \left( -3\,s+2 \right) {s}^{5}}{ \left( s-1 \right) {s}^{
2}}},{\frac {t \left( s+1 \right) + \left( -5\,s+3 \right) {s}^{2}}{
 \left( s-1 \right) {s}^{2}}},{\frac {t \left( s+2 \right) + \left( -8
\,s+5 \right) {s}^{2}}{ \left( s-1 \right) {s}^{2}}}\right).
\]
It parametrizes the degree 4 ruled surface defined by

\noindent $F(x,y,z)=x{y}^{3}-3\,{y}^{2}zx+3\,xy{z}^{2}-x{z}^{3}-2\,{y}^{4}+7\,z{y}^{3}-9\,
{y}^{2}{z}^{2}+5\,y{z}^{3}-{z}^{4}-9\,{y}^{2}x+18\,zyx-9\,{z}^{2}x-5\,
{y}^{3}-9\,{y}^{2}z+16\,{z}^{2}y-5\,{z}^{3}+27\,yx-27\,zx-78\,{y}^{2}+
89\,zy-23\,{z}^{2}-27\,x-14\,y+17\,z+12.$

We observe that $\cP(s,t)$ is in standardized form, so we go to Step 3 in Algorithm \ref{alg-1}. $I$ is the ideal generated by
 $\{p_1t+r_1,p_2t+r_2,p_3t+r_3,q\}$ in $k[s,t]$. We get $\sqrt{I}\neq I$, and a Gr\"obner basis of $\sqrt{I}$ w.r.t. the lexicographic ordering $t>s$ (Step 4) is
 $\{t^2-t, -t+s\}$. So, in Step 4 (b) we get that $f(s)=s$. Note that the affine base points are $(0,0)$ and $(1,1)$ and $t=s$ is the interpolating line; observe that the corresponding Gr\"obner basis of $I$, $\{t^2-t, st-t, s^2-t \}$, that does not read the interpolating polynomial of minimal degree, although it contains the parabola $t=s^2$ that passes through the base points.
  In Step 4 (b, i), we replace $\cP(s,t)$ by
 $\cP(s,1/t+s)$, namely
 \[ \cP(s,t)=\left(-{\frac {3\,{s}^{5}t-2\,{s}^{4}t-ts-1}{st \left( s-1 \right) }},-{
\frac {5\,{s}^{3}t-4\,t{s}^{2}-ts-s-1}{t \left( s-1 \right) {s}^{2}}},\right. \]\[\left.
-{\frac {8\,{s}^{3}t-6\,t{s}^{2}-2\,ts-s-2}{t \left( s-1 \right) {s}^{
2}}}\right). \]
 In Step 4 (b, ii), $\widetilde{Q}=s^2-s$. So, we replace $\cP$ by $\cP(s, 1/((s^2-s)t))$, namely
 \[
  \cP(s,t)=\left( {\frac {-3\,{s}^{4}-{s}^{3}-{s}^{2}+ts-s}{s}},{\frac {ts-5\,s+t-1}{s}
},{\frac {ts-8\,s+2\,t-2}{s}}\right).
\]
Now, the lexicographic order Gr\"obner basis of $\sqrt{I}$ is $\{s, t-1\}$, and hence $\cP$ still have one base point, namely $(0,1)$. Now, the interpolation polynomial is $f(s)=1$ and $\widetilde{Q}=s$. Repeating the steps as above we reach at the end of Step 4
\begin{equation}\label{eq-ejer-3}
 \cP(s,t)=(- (3\,{s}^{2}+s-t+1) s,ts+t-4,ts+2\,t-7)
 \end{equation}
In Step 5 we get
\[ \alpha_{12}=-3{s}^{5}-4{s}^{4}-2{s}^{3}+3{s}^{2},\alpha_{13}=-3{s}^{5}-7{s}^{4}-3{s}^{3}+5{s}^{2},\alpha_{23}=-3{s}^{2}+s
\]
In Step 6 the boolean conditions do  hold, and the output is the parametrization $\cP$ in (\ref{eq-ejer-3}) which is normal.
\end{example}

\section{Removal of Base Points: the General Case}

In Lemma \ref{lemma-reglada-sin-puntos-base} we have seen that, for the special case of standardized ruled surfaces, one can always find a reparametrization such that the new parametrization does not have affine base points. In this section we see that the ideas applied in the proof of that lemma can be generalized to any rational parametrization. More precisely we have the following result.

\begin{theorem}\label{th: remove bpt}
 Let $\cP\colon k^2\dashedrightarrow k^3$ be an affine rational parametrization,  with nonconstant components, of a surface. Then there exists a rational reparametrization $\cP\circ\psi$ without affine base points. Moreover,  $\deg(\cP)=\deg(\cP\circ\psi)$ as rational maps; in particular, properness is preserved.
\end{theorem}

\begin{proof}
 If $\cP$ has no affine base points, take as $\psi$ the identity.

 We can assume without loss of generality that, after a suitable linear birational change,
 \[
  \cP(s,t) = \left( \frac{p_1(s,t)}{q(s,t)}, \frac{p_2(s,t)}{q(s,t)}, \frac{p_3(s,t)}{q(s,t)} \right)
 \]
 where $\deg(p_1)=\deg(p_2)=\deg(p_3)=\deg(q)$, $\gcd(p_1,p_2,p_3,q)=1$, and the projective point $(0:1:0)$ does not belong to any of the projectivizations of the four curves determined by numerators and denominator. We also assume that there are no two affine base points with the same $s$-coordinate, since this can be achieved by composition with $(s,t)\to(s+\lambda t,t)$ for generic $\lambda$ without losing the previous assumptions.

 By the last assumption, there exists an interpolation polynomial $f(s)$ for the affine base points, i.e. for every base point $(s_i,t_i)$ we have $t_i=f(s_i)$; note that the $\gcd$ condition implies finiteness of the base point set. We define the birational reparametrization
 \[
  \psi(s,t)=\left(s,\frac{1}{t}+f(s)\right)
 \]
 and $\tcP=\cP\circ\psi$. We will prove that $\tcP$ has no affine base points. To this end we write
 \[
  \cP=\left(\frac{a_nt^n+a_{n-1}(s)t^{n-1}+\cdots+a_0(s)}{b_nt^n+b_{n-1}(s)t^{n-1}+\cdots+b_0(s)},\ldots,\ldots\right)
 \]
 with $a_n,b_n\neq0$ and $\deg(a_i),\deg(b_i)\leq n-i$. This is possible by the hypothesis on the degrees of $p_1,p_2,p_3,q$, and the fact that $t^n$ appears in all of them with nonzero coefficient (equivalent to the hypothesis on $(0:1:0)$.) Then
 \[
  \tcP=\left(\frac{a_n(1+tf(s))^n+ta_{n-1}(s)(1+tf(s))^{n-1}+\cdots+t^na_0(s)}{b_n(1+tf(s))^n+tb_{n-1}(s)(1+tf(s))^{n-1}+\cdots+t^nb_0(s)},\ldots,\ldots\right).
 \]
 This new parametrization cannot have any base points of the form $(s_0,0)$, since $\tcP(s_0,0)=(a_n/b_n,\ldots,\ldots)$. On the other hand, if $(s_0,t_0)$ is a base point of $\tcP$ with $t_0\neq0$, then $\psi(s_0,t_0)$ is a base point $(s_i,t_i)$ of $\cP$. But this is impossible: if $\psi(s_0,t_0)=(s_i,t_i)$ then $s_0=s_i$ and $1/t_0+f(s_0)=t_i$ which imply $1/t_0=0$, contradiction.   Finally, note that the previous transformations are birational, and $\psi$ is a birational map from $k^2$ on $k^2$, and hence the degree of the parametrization maps is preserved.
\end{proof}

 The reasoning in the previous proof leads to an algorithmic process to remove the affine base points of a surface parametrization. To be more precise, let
 \[
  \cP(s,t) = \left( \frac{p_1(s,t)}{q(s,t)}, \frac{p_2(s,t)}{q(s,t)}, \frac{p_3(s,t)}{q(s,t)} \right)
 \]
 be the surface parametrization. First, we observe that some assumptions on the parametrization are done, namely
\begin{enumerate}
\item {\sf [degree and gcd condition]} $\deg(p_1)=\deg(p_2)=\deg(p_3)=\deg(q)$, and $\gcd(p_1,p_2,p_3,q)=1$,
\item {\sf [condition on $(0:1:0)$]} the projective point $(0:1:0)$ does not belong to any of the projectivizations of the four curves determined by numerators and denominator,
\item {\sf [general position of the base points]} there are no two affine base points with the same $s$-coordinate.
\end{enumerate}

Observe that, in the rational ruled case, condition 3 is satisfied while, in general, conditions 1 and 2 fail because of the particular structure of standardized form, that we wanted to be preserved. So, in Section \ref{sec:covering-theory}, we have developed an \textit{ad hoc} proof for the ruled case.

Once the parametrization satisfies these conditions, one computes the interpolation polynomial $f(s)$ passing through the affine base points.  Then, $\psi(s,t)=(s,\frac{1}{t}+f(s))$. We observe that condition 1 can always be achieved by a birational change of the form
\[ \left(\frac{a_1t+b_1s+c_1}{d_1t+e_1s+h_1}, \frac{a_2t+b_2s+c_2}{d_2t+e_2s+h_2}\right), \]
and condition 2 with a linear change $(s+\lambda t, t)$. In the following lemma we see how to check the third condition and how to actually compute the interpolation polynomial $f(s)$  without approximating roots. This result extends Lemma \ref{lemma-interpolation} to the general case.

\begin{lemma}\label{lemma-interpolacion-general}
Let $I$ be the ideal generated by $\{p_1,p_2,p_3,q\}$ in $k[s,t]$.
\begin{enumerate}
\item Condition 3 is satisfied if and only if there exists a polynomial of the form $t-g(s)$ in $\sqrt{I}$.
\item If $t-g(s)\in \sqrt{I}$, then $g(s)$ interpolates the affine base points.
\end{enumerate}
\end{lemma}
\begin{proof}
If condition 3 holds, then $t-f(s)$ vanishes on all the points in the variety of $I$. So, $t-f(s)\in \sqrt{I}$. The converse is trivial, and (2) follows from (1).
\end{proof}

\begin{algorithm}\label{alg-2}
The input is a rational surface parametrization with affine base points, and the output is a parametrization of the same surface without base points.
\begin{enumerate}
 \item Reparametrize the input to satisfy conditions 1 and 2.
 \item Calculate $\sqrt{I}$; see Step 3 in Algorithm \ref{alg-1}.
 \item Calculate a Gr\"obner basis of $\sqrt{I}$ with respect to the lexicographical ordering $t>s$.
 \begin{enumerate}
  \item If the basis contains a polynomial of the form $t-f(s)$, then by the previous Lemma condition 3 is satisfied and we can apply the reparametrization of Theorem \ref{th: remove bpt} to RETURN $\cP(s,1/t+f(s))$.
  \item In the negative case, by elementary properties of Gr\"obner bases it follows that there is no polynomial of that form in $\sqrt{I}$. Again by Lemma \ref{lemma-interpolacion-general}, condition 3 is not satisfied. Apply a transformation $(s+\lambda t, t)$ for random $\lambda$ in the ground field and go to step 2.
 \end{enumerate}
\end{enumerate}
\end{algorithm}

As a consequence of Theorem \ref{th: remove bpt} and Algorithm \ref{alg-2}, the following corollaries hold.

\begin{corollary}\label{cor-1}
Every rational surface over an algebraically closed field of characteristic zero can always be parametrized without affine base points.
\end{corollary}

\begin{corollary}\label{cor-2}
Every rational surface parametrization can be reparametrized, without affine base points, without extending the field of coefficients and the degree as rational maps.
\end{corollary}

We illustrate the ideas of this section by an example.

\begin{example}\label{ex-pto-base-general}
We consider the rational parametrization
\[ \cP(s,t)=\left(\frac{p_1(s,t)}{q(s,t)},\frac{p_2(s,t)}{q(s,t)}, \frac{p_3(s,t)}{q(s,t)} \right)=  \left({\frac {4\,{s}^{2}-4\,st+{t}^{2}-6\,s+3\,t}{2\,{s}^{2}+8\,st+3\,{t}^{
2}-8\,s-11\,t}},\right. \]
\[\left. {\frac {{s}^{2}-6\,st-{t}^{2}+s+7\,t}{2\,{s}^{2}+8\,st
+3\,{t}^{2}-8\,s-11\,t}},{\frac {-3\,{s}^{2}+22\,st+4\,{t}^{2}-5\,s-26
\,t}{2\,{s}^{2}+8\,st+3\,{t}^{2}-8\,s-11\,t}}\right) \]
Its base points are $\{(0,0),(2,1),(1,2),(1,-1) \}$. We observe that $\cP(s,t)$ satisfies conditions 1 and 2. Let $I$ be the ideal generated
by $\{p_1,p_2,p_3,q\}$. A Gr\"obner basis of $\sqrt{I}$ w.r.t. the lexicographic order with $t>s$ is
\[ \{s^3-3s^2+2s, -s^2+2st+s-2t, 2s^2+t^2-4s-t \}.\]
Since there is no polynomial of the form $t-f(s)$ in the basis,  condition 3 fails, and we perform a change of parameters. For example $\cP$ is replaced by $\cP(s+t,t)$. Applying again the Gr\"obner basis computation to $\sqrt{I}$ for the new $\cP$, we obtain the basis
\[ \{s^4-2s^3-s^2+2s, 2s^3-3s^2-s+2t \}.\]
The second polynomial implies that $t-(-s^3+(3/2)s^2+(1/2)s)\in \sqrt{J}$. So condition 3 is now satisfied and $f(s)=-s^3+(3/2)s^2+(1/2)s$. Therefore, performing the transformation $\cP(s,1/t+f(s))$ we get a new parametrization without affine base points, namely
\[ \left( {\frac {4 {s}^{6}{t}^{2}-12 {s}^{5}{t}^{2}-11 {s}^{4}{t}^{2}+42 {
s}^{3}{t}^{2}-8 {s}^{3}t+7 {s}^{2}{t}^{2}+12 {s}^{2}t-30 {t}^{2}s+
20 st-12 t+4}{52 {s}^{6}{t}^{2}-156 {s}^{5}{t}^{2}+17 {s}^{4}{t}^
{2}+226 {s}^{3}{t}^{2}-104 {s}^{3}t-69 {s}^{2}{t}^{2}+156 {s}^{2}t
-70 {t}^{2}s+100 st-76 t+52}},\right.\]
\[-2 {\frac {12 {s}^{6}{t}^{2}-36 {s
}^{5}{t}^{2}+7 {s}^{4}{t}^{2}+46 {s}^{3}{t}^{2}-24 {s}^{3}t-19 {s}
^{2}{t}^{2}+36 {s}^{2}t-10 {t}^{2}s+20 st-16 t+12}{52 {s}^{6}{t}^
{2}-156 {s}^{5}{t}^{2}+17 {s}^{4}{t}^{2}+226 {s}^{3}{t}^{2}-104 {s
}^{3}t-69 {s}^{2}{t}^{2}+156 {s}^{2}t-70 {t}^{2}s+100 st-76 t+52}
}, \]
\[ \left. {\frac {92 {s}^{6}{t}^{2}-276 {s}^{5}{t}^{2}+51 {s}^{4}{t}^{2}+
358 {s}^{3}{t}^{2}-184 {s}^{3}t-143 {s}^{2}{t}^{2}+276 {s}^{2}t-82
 {t}^{2}s+156 st-124 t+92}{52 {s}^{6}{t}^{2}-156 {s}^{5}{t}^{2}+
17 {s}^{4}{t}^{2}+226 {s}^{3}{t}^{2}-104 {s}^{3}t-69 {s}^{2}{t}^{2
}+156 {s}^{2}t-70 {t}^{2}s+100 st-76 t+52}} \right). \]
\end{example}

\begin{example}
In \cite{Wang}, section 4.5, the author tests his implicitization algorithm with a family of rational surface parametrizations collected from different papers. For those having affine points, we apply Algorithm \ref{alg-2}:
\begin{enumerate}
\item {\sf Example 1 in \cite{Wang}.} The parametrization is
\[ \cP=\left( {\frac {s{t}^{2}-{t}^{3}-t}{{t}^{2}-2\,t+1}},{\frac {{t}^{3}-st-{t}^{
2}+t+1}{{t}^{2}-2\,t+1}},{\frac {st-2\,t}{{t}^{2}-2\,t+1}}\right). \]
 The Gr\"obner basis  of $\sqrt{I}$ w.r.t. the lexicographical ordering $t>s$ is $\{s-2, t-1\}$; indeed $\cP$ has the affine base point $(2,1)$. So the interpolating polynomial is $f(s)=1$. Therefore,  $\cP(s,1/t+1)$ does not have affine base points.
\item {\sf Example 6 in \cite{Wang}.} The parametrization is
\[ \cP=\left( {\frac {s \left( s+t-1 \right) }{{s}^{2}+st+{t}^{2}-1}},{\frac {t
 \left( s+t-1 \right) }{{s}^{2}+st+{t}^{2}-1}},{\frac {s+t-1}{{s}^{2}+
st+{t}^{2}-1}}\right). \]
 The Gr\"obner basis  of $\sqrt{I}$ w.r.t. the lexicographical ordering $t>s$ is $\{s^2-s, s+t-1\}$; indeed $\cP$ has the affine base points $(0,1),(1,0)$. So the interpolating polynomial is $f(s)=1-s$. Therefore,  $\cP(s,1/t+(1-s))$ does not have affine base points.
 \item {\sf Example 9 in \cite{Wang}.} The parametrization is
\[ \begin{array}{lll}
\cP&= &\left( {\dfrac {{s}^{2}t+2\,{t}^{3}+{s}^{2}+4\,st+4\,{t}^{2}+3\,s+2\,t+2}{{s}
^{3}+{s}^{2}t+{t}^{3}+{s}^{2}+{t}^{2}-s-t-1}},\right. \\ \\
& &\left. {\dfrac {-{s}^{3}-2\,s{t}
^{2}-2\,{s}^{2}-st+s-2\,t+2}{{s}^{3}+{s}^{2}t+{t}^{3}+{s}^{2}+{t}^{2}-
s-t-1}},\right.\\ \\
&& \left.{\dfrac {-{s}^{3}-2\,{s}^{2}t-3\,s{t}^{2}-3\,{s}^{2}-3\,st+2\,{
t}^{2}-2\,s-2\,t}{{s}^{3}+{s}^{2}t+{t}^{3}+{s}^{2}+{t}^{2}-s-t-1}}\right). \end{array} \]
 The Gr\"obner basis  of $\sqrt{I}$ w.r.t. the lexicographical ordering $t>s$ is
 $$ \begin{array}{l}
 \{9\,{s}^{6}+8\,{s}^{5}-12\,{s}^{4}+27\,{s}^{3}+34\,{s}^{2}-44\,s-40,\\
1665\,{s}^{5}+382\,{s}^{4}-2152\,{s}^{3}+4939\,{s}^{2}+1540\,s+3288\,t
-4268\}.
\end{array}$$
 So $\cP$ has 6  affine base points, and the interpolation polynomial is
  \[ f(s)=-{\frac {555}{1096}}\,{s}^{5}-{\frac {191}{1644}}\,{s}^{4}+{\frac {269
}{411}}\,{s}^{3}-{\frac {4939}{3288}}\,{s}^{2}-{\frac {385}{822}}\,s+{
\frac {1067}{822}}.
  \]  Therefore,  $\cP(s,1/t+f(s))$ does not have affine base points.
  \item {\sf Example 10 in \cite{Wang}.} The parametrization is
\[ \begin{array}{lll}
\cP&=&\left( {\dfrac {-{s}^{4}+4\,{s}^{3}t-2\,{s}^{2}{t}^{2}+s{t}^{3}+{s}^{2}t-2\,{
t}^{3}}{-{s}^{3}t+6\,{s}^{2}{t}^{2}-3\,s{t}^{3}+{t}^{4}+{s}^{3}-2\,s{t
}^{2}}},\right.\\ \\ && {\dfrac {-{s}^{3}t-2\,{s}^{3}+{s}^{2}t+3\,s{t}^{2}-{t}^{3}}{-{s
}^{3}t+6\,{s}^{2}{t}^{2}-3\,s{t}^{3}+{t}^{4}+{s}^{3}-2\,s{t}^{2}}}, \\ \\
&&
\left.{
\dfrac {-s{t}^{3}+{s}^{3}-4\,{s}^{2}t-s{t}^{2}+6\,{t}^{3}}{-{s}^{3}t+6
\,{s}^{2}{t}^{2}-3\,s{t}^{3}+{t}^{4}+{s}^{3}-2\,s{t}^{2}}}\right).\end{array} \]
 The Gr\"obner basis  of $\sqrt{I}$ w.r.t. the lexicographical ordering $t>s$ is
 $$ \begin{array}{l}
 \{{s}^{6}-7\,{s}^{5}-20\,{s}^{4}+173\,{s}^{3}-27\,{s}^{2}+s,\\-176\,{s}^{
5}+1205\,{s}^{4}+3605\,{s}^{3}-29867\,{s}^{2}+2371\,s+703\,t\}.
\end{array}$$
 So $\cP$ has 6  affine base points, and the interpolation polynomial is
  \[ f(s)={\frac {176}{703}}\,{s}^{5}-{\frac {1205}{703}}\,{s}^{4}-{\frac {3605}
{703}}\,{s}^{3}+{\frac {29867}{703}}\,{s}^{2}-{\frac {2371}{703}}\,s.
  \]  Therefore,  $\cP(s,1/t+f(s))$ does not have affine base points.
\end{enumerate}
\end{example}

\section{Acknowledgements}

This work was developed, and partially supported, by the Spanish \emph{Ministerio de Eco\-no\-m{\'\i}a y Competitividad} under Project MTM2011-25816-C02-01; as well as Junta de Extremadura and FEDER funds (group FQM024). The first and third authors are members of the
Research Group ASYNACS (Ref. CCEE2011/R34). The second author is a member of the research group GADAC (U. of Extremadura).

\end{document}